\documentclass[12pt]{amsart}
\usepackage[linktocpage]{hyperref}
\usepackage{amssymb, paralist, xspace, graphicx, url, amscd, euscript, mathrsfs,stmaryrd,epic,eepic,color}
\usepackage[all]{xy}
\usepackage{amsthm}
\usepackage{longtable}
\usepackage{lipsum} 
\usepackage{lscape} 
\usepackage{slashbox} 
\usepackage{tikz}
\usetikzlibrary{matrix}
\SelectTips{cm}{}

\usepackage[margin=1.5in]{geometry}

\numberwithin{equation}{section}

1
\numberwithin{subsection}{section}

\allowdisplaybreaks[1]

\newtheorem*{namedtheorem}{\theoremname}
\newcommand{\theoremname}{testing}

\newtheorem{theorem}{Theorem}
\newtheorem{proposition}{Proposition}
\newtheorem{proposition-definition}[subsection]
{Proposition-Definition}
\newtheorem{corollary}[theorem]{Corollary}

\theoremstyle{definition}
\newtheorem{definition}[subsubsection]{Definition}

\newtheorem{example}[subsubsection]{Example}
\newtheorem{remark}{Remark}

\newtheorem*{definition*}{Definition}

\newcommand\cA{\mathcal{A}}

\newcommand\cM{\mathcal{M}}

\newcommand\cO{\mathcal{O}}

\newcommand\cT{\mathcal{T}}

\newcommand\cX{\mathcal{X}}

\newcommand\CC{\mathbb{C}}

\newcommand\HH{\mathbb{H}}

\newcommand\PP{\mathbb{P}}
\newcommand\QQ{\mathbb{Q}}
\newcommand\RR{\mathbb{R}}
\renewcommand\SS{\mathbb{S}}

\newcommand\VV{\mathbb{V}}

\newcommand\ZZ{\mathbb{Z}}

\newcommand\fS{\mathfrak{S}}

\newcommand\fW{\mathfrak{W}}

\newcommand\frg{\mathfrak{g}}

\newcommand\fro{\mathfrak{o}}

\newcommand\frs{\mathfrak{s}}

\newcommand\arr{\ifinner\to\else\longrightarrow\fi}

\def\displaytimes_#1{\mathrel{\mathop{\times}\limits_{#1}}}

\def\displayotimes_#1{\mathrel{\mathop{\bigotimes}\limits_{#1}}}

\newdir{ >}{{}*!/-5pt/@{>}}

\newcommand\doublelong[2]{\mathbin{\xymatrix{{}\ar@<3pt>[r]^{#1}
\ar@<-3pt>[r]_{#2}&}}}

\newlength{\ignora}

\newcommand\Sym{\operatorname{Sym}}
\newcommand\GL{\operatorname{GL}}

\newcommand\Char{\textrm{Char}}
\newcommand\SL{\textrm{SL}}

\theoremstyle{plain}
\theoremstyle{definition}

\begin{document}

\title[Character Formula]
{Character Formulas on Cohomology of Deformations of Hilbert Schemes of $K3$ Surfaces}

\author{Letao Zhang}
\address{Department of Mathematics\\
Stony Brook University\\
Stony Brook, NY 11794\\
U.S.A.} \email{letao.zhang@stonybrook.edu}
\begin{abstract}
Let $X$ be a hyperk\"{a}hler manifold deformation equivalent to Hilbert scheme of $n$ points on a $K3$
surface. 
We compute the graded character formula of the generic  Mumford-Tate group representation on the cohomology ring of $X$, and derive a
 generating series for deducing the number of canonical Hodge
classes on $X$.  The formula indicates the number of Hodge classes on $X$ that remain Hodge under any deformation.
\end{abstract}

\maketitle
%
%
%
%

\tableofcontents

\section{Introduction}\label{sec:introduction}

Let $S$ be a $K3$ surface, and $S^{[n]}$  the Hilbert scheme of $n$
points on $S$; each point in $S^{[n]}$ corresponds to a subscheme of $S$ whose
Hilbert polynomial is the constant $n$. We say $X$ is of $K3^{[n]}$-type if $X$ is hyperk\"{a}hler and 
deformation equivalent to $S^{[n]}$.

Denote by $G_X$ the generic Mumford-Tate group of hyperk\"{a}hler manifolds of $K3^{[n]}$-type. The invariants of 
$G_X$ action on $H^*(X,\QQ)$ correspond to the \emph{canonical Hodge
classes} (see Section \ref{sec:groupcoh}), which are Hodge classes that remain Hodge under any deformation.
  Chern classes of the tangent bundle $\cT_X$
are examples of such canonical Hodge classes.

Now we consider the action of $G_X$ on the cohomology of $X$. In particular, we want to compute 
 the characters of $G_X$ representation on the middle
cohomology of $X$, where $X$ is of $K3^{[n]}$-type.

 The lattice of $H^2(X,\ZZ)$ -- with respect to the Beauville-Bogomolov form --  is 
$\Lambda_S\oplus\delta\ZZ$. Here $\Lambda_S:= U^3\oplus E_8^2(-1)$  is the lattice of $H^2(S,\ZZ)$ where $S$ is a $K3$ surface, and $\left(\delta,\delta\right) = -2(n-1)$.
 Let $G_S$ be the identity component of  $O^+(H^2(S,\QQ))$ with respect to the intersection form. The action of maximal torus $T_X$ of 
$G_X$  on $\Lambda_S$ is the same as the action of maximal torus  $T_S$ of $G_S$ on $\Lambda_S$.

\begin{theorem}\label{thm:middim}
Let $M(q) :=\sum_{n=0}^{\infty}\Char(H^{2n}(X,\QQ))\cdot q^n $ be
the generating series for the character of the $G_X$  representation on the middle cohomology of
$X$. 
$$M(q)= \left({1+\sum_{k=1}2(-1)^k
q^{\frac{k(k+1)}{2}}}\right) \left({\prod_{m=1}\det(I_{24}-gq^m)}
\right)^{-1},$$
where $g\in T_X$ the maximal torus of $G_X$, $I_N$ is a $N\times N$ identity
matrix, and $\det(I_{24}-gt^m)=(1-t^m)^2\det(I_{22}-g|_{T_S}t^m)$.
\end{theorem}

\begin{example}\label{eg:hodge7}
Let $X$ be of $K3^{[7]}$-type. 
There are 7 Hodge classes in $H^{14}(X,\QQ)$ that remain Hodge under any deformation.
Similarly, there are 5 Hodge classes in $H^8(X,\QQ)$,  5 Hodge classes in $H^{10}(X,\QQ)$,
and 10 Hodge classes  in $H^{12}(X,\QQ)$ 
 that  remain Hodge under any deformation. (cf. Appendix \ref{app:t1})
\end{example}

Now let $l\in H_2(X,\ZZ)$ be a line class in
$\PP^{n}\subset X$. Hassett and Tschinkel in
 \cite{HaTs10} show that  $(l,l)=-\frac{5}{2}$
 for the case where $n=2$ in \cite{HaTs09}.
  For $n=3$, Harvey, Hassett and Tschinkel \cite{Br10} show that $(l,l)=-3$
  and give a concrete expression for the Lagrangian hyperplane class. For the case
of $n=4$, Bakker and Jorza \cite{BAJO11} show that
$(l,l)=-\frac{7}{2}$, and also give an expression for $[\PP^4]$. For $n\geq 5$, Bakker \cite{BA13} 
shows that $(l,l) = -\frac{n+3}{2}$, which was conjectured in \cite{HaTs10}. However,
it is more difficult to compute the class $[\PP^n]$ for larger $n$.
 One possible approach to exploring the expression for $[\PP^n]$  is to find all the canonical Hodge classes in the middle
cohomology of $X$ for each $n$; future work could provide possible
candidates for the class of $[\PP^{n}]$ in terms of the line class.
As for the ring structure, Verbitsky \cite{Ver95} 
 shows that there
is an embedding $\Sym^nH^2(S,\QQ)\hookrightarrow
H^{2n}(S^{[n]},\QQ)$, but much about the ring structure of
$H^*(X,\QQ)^{G_X}$ is still unknown, e.g. relations in the
subalgebra generated by $H^*(S,\QQ)$ for each $H^*(X,\QQ)^{G_X}$.\\

 {\bf Acknowledgments}.
 I am very grateful to my advisor Brendan Hassett for introducing me to this
 problem, and for his warm support and encouragement.
 Thanks are also due to Lothar G\"{o}ttsche, Radu Laza, and Anthony V\'{a}rilly-Alvarado
 for interesting discussions and insightful
 remarks. I  appreciate Eyal Markman's illuminating questions which may lead to future research topics. \\
The writing of this paper was supported in part by NSF grant 0901645 and
0968349.
%
%
%
%
\section{Cohomology of Hilbert Schemes of Points on $K3$ surfaces}\label{sec:2}
In this section, we review some classical results about
$S^{[n]}$, where $S$ is a $K3$ surface. For $n>1$, the Beauville-Bogomolov form can be written as the direct sum \cite{Be83}
\begin{equation}\label{eq:decomp2,2}
 H^2(S^{[n]},\ZZ)= H^2 (S, \ZZ)_{\left(~,~\right)} \oplus_{\perp} \ZZ\delta,
 ~(\delta,\delta)=-2(n-1),
\end{equation}
 where $\left(~,~\right)$ is the intersection form on $H^2 (S, \ZZ)$, and
  $2\delta$ is the class of the corresponding big diagonal divisor
$\Delta^{[n]}\subset S^{[n]}$ parameterizing nonreduced subschemes.

In \cite{Heisen}, Nakajima constructs generators for the cohomology
ring of Hilbert schemes of points of any projective surface. Lehn
and Sorger  \cite{LS03} then show how $H^*(S,\QQ)$ generates
$H^*(S^{[n]},\QQ)$ as a graded ring. 

Let $A=H^*(S,\QQ)[2]$ denote the shifted cohomology ring weighted by
-2, 0, 2. Correspondingly, let $\HH_n=H^*(S^{[n]},\QQ)[2n]$ denote
the shifted cohomology ring weighted by $-2n,...,2n$. Note that the
weight shifting here is not the Tate twist notation for Hodge
classes.

Define a linear form $T$ on $A$ by $T(a):=-\int_{[S]}a$, and let $\left<,\right>$ be the 
induced bilinear form on the shifted cohomology $\left<a_1,a_2\right> = T(a_1a_2) = -\int_S a_1a_2$.
On $A^{\otimes n}$ , 
 one  can define an analogous structure.  Since $A$ and $\HH_n$ have only graded pieces of \textit{even} 
weights, we can simplify the algebraic model in \cite{LS03}.

The product is given by
$$(a_1\otimes\cdots\otimes a_n)\cdot (b_1\otimes\cdots\otimes b_n)=(a_1b_1)\otimes\cdots\otimes(a_nb_n)~.$$
$T$ extends to $A^{\otimes n}$ via
$$T(a_1\otimes\cdots\otimes a_n)=T(a_1)\cdot\cdots\cdot T(a_n)~,$$
and the bilinear form $\left<,\right>$ on $A^{\otimes n}$ is defined accordingly:
$$\left<a,b\right> = T(a)T(b).$$
We also have the symmetric group $\fS_n$ action on the $n-$fold tensor given
by
$$\pi(a_1\otimes\cdots\otimes a_n)=a_{\pi^{-1}(1)}\otimes\cdots\otimes a_{\pi^{-1}(n)}~.$$
For any partition $n=n_1+\cdots+n_k$, we have a homomorphism
\begin{align*}
A^{\otimes n}                   &\rightarrow A^{\otimes k}\\
a_1\otimes\cdots\otimes a_n     &\mapsto(a_1\cdots
a_{n_1})\otimes\cdots\otimes(a_{n_1+\cdots+n_{k-1}+1}\cdots a_{n_k})
\end{align*}
Given a finite set $I$ with $n$ elements, let $\{A_i\}_{i\in I}$ be
a family of copies of $A$ indexed by $I$. Let $[n]$ denote
$\{1,\dots,n\}$; we define
$$A^{\otimes I}:=\left(\bigoplus_{f:[n]\xrightarrow\cong I}A_{f(1)\otimes\cdots\otimes f(n)}\right)/\fS_n$$
Finally, given a surjection $\phi:I\rightarrow J$ between two index
sets, there is an induced multiplication
$$\phi^*:A^{\otimes I}\rightarrow A^{\otimes J}~,$$
 and let
$$\phi_*:A^{\otimes J}\rightarrow A^{\otimes I}$$ be the \textit{adjoint} of
$\phi^*$, i.e. 
$$\left<\phi^*a,b\right>= \left<a,\phi_* b\right>,$$
where $a\in A^{\otimes I}, b\in A^{\otimes^J}$. The projection formula
$$\phi_*(a\cdot \phi^*(b))=\phi_*(a)\cdot b~$$
 holds by \cite{LS03}.

Denote by $\left<\pi\right>\backslash[n]$ the set of orbits of $[n]$
under the action of $\pi$. Define
$$A\{\fS_n\}:=\oplus_{\pi\in\fS_n}A^{\otimes\left<\pi\right>\backslash[n]}\cdot\pi$$
 $A\{\fS_n\}$ admits an action of $\sigma\in\fS_n$, induced by the bijection
$$\sigma: \left<\pi\right>\backslash[n]
\rightarrow\left<\sigma\pi\sigma^{-1}\right>\backslash[n],~~ x
\mapsto\sigma x. $$ This gives an automorphism of $A\{\fS_n\}$ given
by
$$\widetilde{\sigma}: a\cdot\pi
\mapsto \sigma^*(\sigma\pi\sigma^{-1}).$$ Denote by $A^{[n]}$ the
invariants under this action, then we have the graded isomorphism
between the vector spaces \cite{LS03}
$$A^{[n]}=\sum_{\|\alpha\|=n}\bigotimes_i\Sym^{\alpha_i}A$$
where $\alpha=(1^{\alpha_1},2^{\alpha_2},\dots,n^{\alpha_n})$ runs
all partitions of $n$ and $\|\alpha\|=\sum_{i=1}^{n}i\alpha_i$. For the case of $K3$ surfaces,
Lehn and Sorger prove
\begin{theorem}\label{thm:lehnsoger}\cite{LS03} There is a canonical isomorphism of
graded rings
$$(H^*(S,\QQ)[2])^{[n]} \xrightarrow\cong H^*(S^{[n]},\QQ)[2n]~.$$
\end{theorem}
%
    
%
%
%
%
\section{Decomposition of Cohomology Ring}
In this section, we first review some useful results from representation theory. We then discuss generic Mumford-Tate group actions on the cohomology ring of $X$ of $K3^{[n]}$-type. Finally, we introduce canonical Hodge classes as the invariants of the group action.  Our goal is to decompose $H^*(X,\QQ)$ into irreducible representations and to count invariants.

\begin{subsection}{Characters of Representations}
We summarize general results on representations of complex (or
split) orthogonal groups from \cite{rep}.

Let $\frg$ be a semisimple Lie algebra, $\Lambda$ be its weight
lattice, and $\ZZ[\Lambda]$ be the integral group ring of the
abelian group $\Lambda$. For each weight $\lambda\in\Lambda$, let
$e(\lambda)$ denote the basis element in $\ZZ[\Lambda]$, so that each
element in $\ZZ[\Lambda]$ can be written as the finite sum $\sum_\lambda
n_\lambda\cdot e(\lambda)$. Denote by $R(\frg)$ the ring of
isomorphism classes of finite-dimensional representations associated
to $\frg$. For each class $[V]$, $[V]=[V']+[V'']$ whenever
$V=V'\oplus V''$, and the product of two classes is defined as
$[V]\cdot[W]=[V\otimes W]$. Define the character homomorphism
\[\Char:R(g)\rightarrow\ZZ[\Lambda]\]
by $\Char[V]=\sum\#(V_\lambda)\cdot e(\lambda)$, where $V_\lambda$
is the weight space of $V$ for the weight $\lambda$ and
$\#(V_\lambda)$ is the multiplicity of $V_\lambda$ in $V$.  The Weyl
group $\fW$ acts on $\ZZ[\Lambda]$ and the image of $\Char$ is
contained in the ring of invariants $\ZZ[\Lambda]^\fW$.

Let $\omega_1,\dots,\omega_n$ be fundamental weights of $\frg$.
Recall that fundamental weights have the property that any highest
weight may be expressed uniquely as a nonnegative integral linear
combination of them; they are free generators for the lattice
$\Lambda$. Let $\Gamma_i~(i=1,\dots,n)$ be the classes in $R(\frg)$
of the irreducible representations of highest weight
$\omega_i~(i=1,\dots,n)$. We have the following theorem.
\begin{theorem}\label{thm:character}\cite{rep}
The representation ring $R(\frg)$ is a polynomial ring on the
variables $\Gamma_1,\dots,\Gamma_n$, and the homomorphism
$\Char:R(\frg)\rightarrow\ZZ[\Lambda]^\fW$ is an isomorphism.
\end{theorem}
Thus decomposing $V$ into irreducible $\frg$ representations is
equivalent to finding its character polynomial.

\begin{example}\cite{rep}\label{eg:so22}
Let $\frg=\frs\fro_{2n}\CC$ and $V\cong\CC^{2n}$ be its standard
representation. Its weight lattice $\Lambda$ is
span$\left\{L_1,\dots,L_{n},\left(\sum L_i\right)/2\right\}$
(see Lecture 19 in \cite{rep} for detailed explanation). For
$\frs\fro_{2n}\CC$, fundamental weights are
$$L_1, L_1+L_2,\dots,L_1+\cdots+L_{n-2}, (L_1+\cdots+L_{n})/2, (L_1+\cdots+L_{n-1}-L_{n})/2$$
corresponding to irreducible representations $V,\bigwedge^2
V,\dots,\bigwedge^{n-2}V$ and the half-spin representations $S^+$
and $S^-$.
Set $t_i=e(L_i)$, $t_i^{-1}=e(-L_i)$, $t_i^{+1/2}=e(L_i/2)$,
$t_i^{-1/2}=e(-L_i/2)$, $\Char(\bigwedge^k V)$ is the $k$-th
elementary symmetric polynomial -- denoted by $D_k$ -- of the $2n$
elements $t_1$, $t_1^{-1}$, $\dots$, $t_{n}$, $t_{n}^{-1}$. The
character $D^+~(\textrm{resp. }D^-)$ of $S^+~(\textrm{resp. }S^-)$
is the sum $\sum t_1^{\pm1/2}\cdot\dots\cdot t_n^{\pm1/2}$, where
the number of plus signs is even (\textrm{resp. }odd). Thus,
$$R(\frs\fro_{2n}\CC)=\ZZ[\Lambda]^{\fW}=\ZZ[D_1,\dots,D_{n-2},D^+,D^-]$$
\end{example}
\begin{example}\cite{rep}
In the case of $\frg=\frs\fro_{2n+1}\CC$, its standard
representation is $V\cong\CC^{2n+1}$ and its weight lattice is the
same as $\frs\fro_{2n}\CC$. But the fundamental weights are
$$L_1, L_1+L_2,\dots,L_1+L_2+\cdots+L_{n-1}, (L_1+\cdots+L_{n})/2$$
corresponding to irreducible representations $V,\bigwedge^2
V,\dots,\bigwedge^{n-1}V$ and the spin representation $S$.
$\Char(\bigwedge^k V)$ here is the $k$-th elementary symmetric
polynomial -- denoted by $B_k$ -- of $2n+1$ elements $t_1,
t_1^{-1},\dots,t_{n},t_{n}^{-1}$ and $1$. Denote by $B_n$ the
character of $S$, which is the $n$-th symmetric polynomial in
variables $t_i^{\frac{1}{2}}+t_i^{-\frac{1}{2}}$.
By applying Theorem \ref{thm:character} we obtain
$$R(\frs\fro_{2n+1}\CC)=\ZZ[\Lambda]^{\fW}=\ZZ[B_1,\dots,B_{n-1},B_n]$$
\end{example}
If
$\Gamma_{\lambda}$ is an irreducible $\frs\fro_{2n+1}\CC$
representation of
 highest weight
$\lambda=(\lambda_1\geq\cdots\geq\lambda_n\geq 0)$, then its image in $\ZZ[B_1,\dots,B_n]$
is $B_1^{\lambda_1-\lambda_2}B_2^{\lambda_2-\lambda_3}\cdots B_{n-1}^{\lambda_{n-1}-\lambda_n}B_n^{\lambda_n}$. 
In general, we have $\frs\fro_{2n}\CC\subset\frs\fro_{2n+1}\CC$, and the restriction
representation of $\Gamma_{\lambda}$ is
$$\textrm{Res}^{\frs\fro_{2n+1}\CC}_{\frs\fro_{2n}\CC}\Gamma_{{\lambda}}=\oplus_{\bar{\lambda}}\Gamma_{\bar{\lambda}}$$
where $\bar{\lambda}=(\bar{\lambda}_1,\dots,\bar{\lambda}_n)$
satisfies
$${\lambda}_1\geq\bar{\lambda}_1\geq{\lambda}_2\geq\bar{\lambda}_2
\geq\cdots\geq\bar{\lambda}_{n-1}\geq{\lambda}_n\geq|\bar{\lambda}_n|~,$$
and $\bar{\lambda}_i$ and ${\lambda}_i$ are either all integers
or all half integers.

%
Given a finite dimensional $\frs\fro_{2n+1}\CC$ representation $W$, if it
 is induced by the inclusion
$\frs\fro_{2n}\CC\subset\frs\fro_{2n+1}\CC$ and all the weights of
$\frs\fro_{2n}\CC$ representation are integer-valued, then so are
the weights of the $\frs\fro_{2n+1}\CC$ representation. This implies
that the character of the $\frs\fro_{2n+1}\CC$ representation will be in
$\ZZ[B_1,\dots,B_{n-1}]$.

\end{subsection}

\begin{subsection}{Group actions on cohomologies}\label{sec:groupcoh}
Let $V$ be a $\QQ$ vector space, and $\SS(\RR)\cong \CC^*$  regarded as a Lie group. $ S^1$ 
 is a maximal compact subgroup
of $\SS(\RR)$.

\begin{definition}\label{def:Hdg}
(\cite{GR12}, I.A)
A \textit{Hodge structure} of weight $n$ is given by 
a representation on $V_\RR:=V\otimes_\QQ \RR$
$$\tilde{\varphi}:\SS(\RR)\rightarrow\GL(V_\RR)$$
such that for $r\in\RR^*\subset\SS(\RR)$, $\tilde{\varphi}(r) = r^n id_V$.

\noindent This definition is equivalent to giving a Hodge decomposition of $V_{\CC} := V\otimes_{\QQ} \CC$, where
$$V_\CC=\oplus V^{p,q},~\textrm{and  } V^{p,q}=\bar{V}^{q,p}.$$

\noindent For $\varphi:=\tilde{\varphi}|_{S^1}$, we obtain a representation
        $$\varphi:S^1\rightarrow\SL(V)(\RR)$$ given by
        $\varphi(e^{i\theta})v=e^{i\theta(p-q)}v$, for $v\in
        V^{p,q}$.
\end{definition}

\begin{definition}\cite{GR12}
The \textit{Mumford-Tate group} $M_\varphi$ associated to a Hodge
structure $(V,\varphi)$ of weight $n$ is the  $\QQ$-algebraic closure
of \[\varphi:S^1\rightarrow\SL(V_\RR).\]
\end{definition}

\noindent For a pure Hodge structure of even weight $2p$, {\em Hodge classes} are
defined as those lying in the intersection  $V\cap V^{p,p}_{\CC}$. Let $T^{k,l}=V^{\otimes k}\otimes\check{V}^{\otimes l}$ and $Hg(V_\varphi)$ denote the direct sum of 
all Hodge classes in $T^{k,l}$ for all pairs of $(k,l)$. It is known that 
$M_\varphi$ fixes $Hg(V_\varphi)$  \cite{GR12} .

Let $V$ denote the weight two Hodge structure on $H^2(X,\ZZ)$,  Theorem 
2.2.1 in \cite{ZY83} indicates that
\begin{proposition}\label{prop:MTso}
	For generic $X$, the Mumford-Tate group $G_X=SO\left(V,\left<,\right>\right)$, where $\left<,\right>$ is the Beauville-Bogomolov form. 
\end{proposition}

\begin{definition}\label{def:monodromy}\cite{Ma08, Ma11} An automorphism $g$ of
 $H^*(X,\QQ)$ is called
a \textit{monodromy operator} (equivalently, a parallel transport operator)  if there exists a 
smooth and proper family  
$\cM\rightarrow B$ (which may depend on $g$) of irreducible
holomorphic symplectic manifolds over a (possibly singular) complex analytic space $B$,
having $X$ as a fiber over a point $b\in B$, and such that $g$
belongs to the image of $\pi_1(B, b)$ under the monodromy
representation. The monodromy group $Mon(X)\in GL(H^*(X,\QQ))$ is generated
by all the monodromy operators. In this context, the \textit{algebraic monodromy group} 
$\overline{Mon}(X)$ is defined as the smallest $\QQ$-algebraic group in $GL(H^*(X,\QQ))$ that contains 
$Mon(X)$. Denote by $\overline{Mon}^2(X)$ the image of $\overline{Mon}(X)$ in the isometry group 
of $H^2(X,\QQ)$.
\end{definition}
\begin{proposition}\label{prop: gequiv}
Let $X$ be of $K3^{[n]}$-type, and assume $X$ is a very general fibre  in the universal family $\cX\rightarrow B$. 
Let $\overline{Mon}^2_0(X)$ be the identity component of $\overline{Mon}^2(X)$, then we have $\overline{Mon}^2_0(X) = G_X$.
\end{proposition}
\begin{proof}
 Theorem 16 in \cite{PE03} shows that 
any connected component of $\overline{Mon}^2(X)$ is a normal subgroup of the derived group 
of $G_X$.  In particular, we have  $\overline{Mon}^2_0(X) \subset G_X$. \\
Consider the map $\iota: O^+\left(H^2(X,\ZZ)\right)\rightarrow O\left(H^2(X,\ZZ)^*/H^2(X,\ZZ)\right)$, 
  Lemma 4.2 in \cite{MA10} shows that 
 $Mon^2(X)$ is the inverse image of the subgroup $\{1,-1\}$ under  $\iota$. Thus  $SO\left(V,\left<,\right>\right)\subset\overline{Mon}^2_0(X)$. 
 Since $G_X=SO\left(V,\left<,\right>\right)$ by Proposition \ref{prop:MTso}, 
$\overline{Mon}^2_0(X) = G_X$.
\end{proof}
\begin{example}\label{eg:k3}
The simplest case is when $X$ is of $K3$-type. The monodromy group $\overline{Mon}^2(X)$ is $\cO^+(H^2(X,\QQ))$ 
\cite{BO86}. Thus its identity component $\overline{Mon}^2_0(X)$ is $SO(H^2(X,\QQ))$.
\end{example}

\begin{definition}
Let $X$ be of $K3^{[n]}$-type, and assume $X$ is a very general fibre in the universal family $\cX\rightarrow B$.  
{\em Canonical Hodge classes} of $X$ are Hodge classes
that remain Hodge under any  deformation.
\end{definition}

\begin{theorem}\label{thm:canonical}
Let $X$ be of $K3^{[n]}$-type, and assume $X$ is a very general fibre over $b$ in the universal family $\cX\rightarrow B$.  
Canonical Hodge classes of $X$ are exactly the invariant classes  $H^*(X,\QQ)^{G_X}$. 
\end{theorem}
\begin{proof}
Since canonical Hodge classes are Hodge classes, they are contained in $H^*(X,\QQ)^{G_X}$. 

Given any very general fibre $X'$, let $\gamma\subset B$ be a path such that 
$\gamma(0) = b$ and $\gamma(1) = b'$ where $\cX_{b'} = X'$.
Markman \cite{Ma11} shows that $\overline{Mon}^2(X)$ is a normal subgroup of $\overline{Mon}(X)$, 
thus $\gamma\overline{Mon}^2_0(X)\bar{\gamma}\subset \overline{Mon}^2(X')$. Since $\gamma\overline{Mon}^2_0(X)\bar{\gamma}$ is connected and 
contains identity, $\gamma\overline{Mon}^2_0(X)\bar{\gamma}=\overline{Mon}^2_0(X') = G_{X'}$. 

  For any
 $\alpha\in H^*(X,\QQ)^{G_X}$, $\gamma\alpha\in  H^*(X',\QQ)$.  Given $h'\in G_{X'}$, there exists $h\in G_X$ such that $h' = \gamma h\bar{\gamma}$. This implies 
\begin{align*}
h'(\gamma\alpha) &= \gamma h\bar{\gamma}(\gamma\alpha) \\
	&= \gamma h \alpha \\
	&= \gamma\alpha
\end{align*}
\noindent Thus $\gamma\alpha\in H^*(X',\QQ)^{G_{X'}}$ is a Hodge class.

Now let $\cA\in H^*(\cX,\QQ)$ be the class such that $\cA_b = \alpha\in H^*(\cX_b,\QQ)$.  For any very general fibre $\cX_{b'}$,  $\cA_{b'}\in H^*(\cX_{b'},\QQ)$ is obtained by a path from $b$ to $b'$. Passing through finite \'{e}tale cover of $B$, $\cA_{b'}$ is unique. By the above analysis,  $\cA_{b}$ are Hodge for all very general fibres.  By Deligne-Cattani-Kaplan theorem in \cite{DE95},  Hodge loci of $\cA$ is closed. Thus  $\cA_{s}$ are Hodge for all special fibres $\cX_s$. Thus $\alpha$ remains Hodge under any deformation.

Thus  $H^*(X,\QQ)^{G_X}$ is the collection of all canonical Hodge classes. 
\end{proof}

Note that the same statement was proved in Lemma 3.2 of \cite{MA12}.

By the above analysis,  the invariants of the $G_X$ representation
on $H^*(X,\QQ)$ are the canonical Hodge classes. By computing 
number of trivial  $G_X$ representations on $H^{2p}(X,\QQ)$, we can obtain 
number of canonical Hodge classes of type $(p,p)$ of $X$.
\begin{example}
In Table \ref{tab:invariants}, the first row of data shows the number of
canonical Hodge classes. For instance , if $X$ is of 
$K3^{[5]}$-type, then there are 2 canonical Hodge classes in $H^{2,2}(X)$, 1 in $H^{3,3}(X)$,
 4 in $H^{4,4}(X)$, and 2 in $H^{5,5}(X)$.

\end{example}

%
\end{subsection}


\begin{subsection}{Decomposition of the Cohomology Representation}\label{sec:cohRep}
Denote by $G_S$ the identity component of the special orthogonal
group associated with the intersection form on $H^2(S,\ZZ)$. (cf. Example \ref{eg:k3})

 We can decompose $H^*(X,\QQ)$ into irreducible representations for the
action of $G_X$.  \cite{Br10} provides an explicit method for writing the
decomposition:
\begin{enumerate}
	\item Use the isomorphism
	$H^2(X,\ZZ)\cong\Lambda_S\oplus_\perp \ZZ\delta$ and compatible
maximal torus of $G_S$ and $G_X$ to fix the embedding $G_S\subset
G_X$.  (cf. Section \ref{sec:introduction})
\item Decompose $H^*(S^{[n]},\QQ)$ into the highest-weight $G_S$ representations 
 using $H^*(S^{[n]},\QQ)[2n]\cong \left( H^*(S,\QQ)[2]\right)^{[n]}$ in Theorem \ref{thm:lehnsoger}.
 The highest-weight irreducible
representation $V_S(\lambda)$ will lie in the summand of an irreducible $G_X$ representation $V_X(\lambda)$ in $H^*(X,\QQ)$. 

\item Repeat step 1 and 2 on $H^*(X,\QQ)/V_X(\lambda)$.
\end{enumerate}

Let $\VV_{2k,n}:=H^{2k}(X,\QQ)$ where  $X$ of $K3^{[n]}$-type, $\VV_{2k,n}$ is of weight $2k-2n$ in 
$H^*(X,\QQ)[2n]$. Let $V_{\lambda}$ be the irreducible $G_{X}$ representation of the highest weight $\lambda$;
 we get the following computational results
\begin{longtable}{ |   l | l | l | l | l | l | l | l | l | l | l }
    \hline
    \tiny$\lambda$  	     & \tiny\tiny$\textrm{dim} V_{\lambda}$ &\tiny $V_{4,5}$    &\tiny $\VV_{6,5}$  & \tiny $\VV_{8,5}$ & \tiny $\VV_{10,5}$& \tiny $\VV_{6,6}$ & \tiny $\VV_{8,6}$   & \tiny $\VV_{10,6}$  & \tiny $\VV_{12,6}$   & \tiny\tiny ...\\ \hline 
    \tiny$(0,0,0,\dots)$ & \tiny1                           &\tiny 2        &\tiny 1          & \tiny 4         & \tiny 2         & \tiny 2         & \tiny 5           & \tiny 4           & \tiny 7            & \tiny\tiny ...\\ \hline 
    \tiny$(1,0,0,\dots)$ & \tiny23                          &\tiny 1        &\tiny 3          & \tiny 3         & \tiny 5         & \tiny 3         & \tiny 4           & \tiny 7           & \tiny 7            & \tiny\tiny ...  \\ \hline 
    \tiny$(2,0,0,\dots)$ & \tiny275                         &\tiny 1        &\tiny 1          & \tiny 3         & \tiny 2         & \tiny 1         & \tiny 4           & \tiny 4           & \tiny 7            & \tiny\tiny ...  \\ \hline 
    \tiny$(1,1,0,\dots)$ & \tiny253                         &           &\tiny 1          & \tiny 1         & \tiny 2         & \tiny 1         & \tiny 1           & \tiny 3           & \tiny 2            & \tiny\tiny ... \\ \hline 
    \tiny$(3,0,0,\dots)$ & \tiny2277                        &           &\tiny 1          & \tiny 1         & \tiny 2         & \tiny 1         & \tiny 1           & \tiny 3           & \tiny 3            & \tiny\tiny ... \\ \hline 
    \tiny$(2,1,0,\dots)$ & \tiny4025                        &           &\tiny            & \tiny 1         & \tiny 1         & \tiny           & \tiny 1           & \tiny 2           & \tiny 2            & \tiny\tiny ...  \\ \hline 
    \tiny$(1,1,1,\dots)$ & \tiny1771                        &           &\tiny            & \tiny           & \tiny           & \tiny           & \tiny             & \tiny             & \tiny 1            & \tiny\tiny ... \\ \hline 
    \tiny$(4,0,0,\dots)$ & \tiny14674                       &           &\tiny            & \tiny 1         & \tiny           & \tiny           & \tiny 1           & \tiny 1           & \tiny 2            & \tiny\tiny ... \\ \hline 
    \tiny$(3,1,0,\dots)$ & \tiny256795                      &           &\tiny            & \tiny           & \tiny 1         & \tiny           & \tiny             & \tiny 1           & \tiny 1            & \tiny\tiny ... \\ \hline 
    \tiny$(2,2,0,\dots)$ & \tiny2193763                     &           &\tiny            & \tiny           & \tiny           & \tiny           & \tiny             & \tiny             & \tiny 1            & \tiny\tiny ... \\ \hline 
    \tiny$(5,0,0,\dots)$ & \tiny7804350225                  &           &\tiny            & \tiny           & \tiny 1         & \tiny           & \tiny             & \tiny  1          & \tiny              & \tiny\tiny ... \\ \hline 
    \tiny$(4,1,0,\dots)$ & \tiny$\cdots$                &           &\tiny            & \tiny           & \tiny           & \tiny           & \tiny             & \tiny  $\cdots$          & \tiny $\cdots$            & \tiny\tiny ... \\ \hline 
    \tiny$(6,0,0,\dots)$ & \tiny$\cdots$                    &           &\tiny            & \tiny           & \tiny           & \tiny           & \tiny             & \tiny             & \tiny $\cdots$            & \tiny\tiny ... \\ \hline 
\caption{$G_X$ Representations} 
\label{tab:invariants}
\end{longtable}
\noindent Here ``$\cdots$"  denotes truncated data, and each integer denotes the
number of times $V_{\lambda}$ appears in $\VV_{2k,n} = H^{2k}(X,\QQ)$. In
particular, the first row in the table indicates the number of
copies of trivial $G_X$ representation in each $H^{2k}(X,\QQ)$,
and corresponds to the number of canonical Hodge classes.

$H^2(S,\QQ)$ corresponds to the standard $G_S$ representation
$V_S(1)$.
Let $\HH_n =H^*(S^{[n]},\QQ)[2n] $, which is a
bigraded algebra associated to $H^*(S^{[n]},\QQ)$.   By Theorem
\ref{thm:lehnsoger}, we can decompose $\HH_n$ into $G_S$
representations. By using the same convention as
Example \ref{eg:so22}, we know that
$$\Char_{G_S}(H^{2k}(S^{[n]},\RR))\in \ZZ[D_1,\dots,D_{9},D^+,D^-]~.$$
\begin{remark}	
	When $k\leq 10$,  spin representations do not appear in our
	representation as $\wedge^k V_S(1,0,\dots)$ is irreducible (Theorem 19.2 \cite{rep}). 
\end{remark}
	By compatibility of maximal tori of $G_X$ and $G_S$ (cf. Section \ref{sec:introduction}),  we have
 	\begin{align}\label{eq:char}
  	\Char_{G_S}(H^{2k}(S^{[n]},\RR)) &= \Char_{G_X}(\VV_{2k,n})
 	\end{align}

By \ref{eq:char}, we will not distinguish notations between $\Char_{G_S}$ and $\Char_{G_X}$, and will denote both by
$\Char$ in the
following discussion.
Let $$p(z)_n:=\sum_{k=-n}^{n}\Char(\VV_{2k+2n,n})\cdot
z^{2k}\in\ZZ[D_1,\dots,D_{9},D^+,D^-][z,\frac{1}{z}]$$ be the graded
character of $H^*(X,\RR)$. Now we take the sum
\begin{align}\label{eq:pzt}
p(z,t) &:= \sum_{n=0}^{\infty}p(z)_nt^n
\end{align}
following the grading of each
$\VV_{\bullet, n}$.
\begin{example}\label{eg:char}
Consider $X$  of $K3^{[3]}$-type, and $\HH_3=H^*(S^{[3]},\QQ)[6]$.
By the above analysis, we obtain
\begin{align*}
H^0(S^{[3]},\QQ)[6]&=1_S\\
H^2(S^{[3]},\QQ)[6]&=1_S\oplus V_S(1)\\
H^4(S^{[3]},\QQ)[6]&=1_S^3\oplus V_S(1)^2\oplus V_S(2)\\
H^6(S^{[3]},\QQ)[6]&=1_S^3\oplus V_S(1)^3\oplus V_S(2)\oplus V_S(1,1)\oplus V_S(3)
\end{align*}
Thus we have
\begin{align*}
 p(z)_3 =& \sum_{k=-3}^{3}\Char_{G_S}(H^{6+2k}(S^{[3]}))\cdot z^{2k}\\
        =& ~z^{-6}+(1+D_1) z^{-4}+(3+2D_1+D_1^2-D_2)z^{-2}\\
         & +(3+3D_1+D_1^2+D_1^3-2D_1D_2+D_3)\\
         & +z^{6}+(1+D_1) z^{4}+(3+2D_1+D_1^2-D_2)z^{2}
\end{align*}
The induced $G_X$ representations are
\begin{align*}
\VV_{0,3}=H^0(X,\QQ)[6]&=1_X\\
\VV_{2,3}=H^2(X,\QQ)[6]&=V_X(1)\\
\VV_{4,3}=H^4(X,\QQ)[6]&=1_X\oplus V_X(1)\oplus V_X(2)\\
\VV_{6,3}=H^6(X,\QQ)[6]&=1_X\oplus V_X(1)\oplus V_X(1,1)\oplus V_X(3)\\
\end{align*}
The character formula is
\begin{align*}
 p(z)_3 =& \sum_{k=-3}^{3}\Char_{G_X}(\VV_{2k+6,3})\cdot z^{2k}\\
        =& z^{-6}+B_1 z^{-4}+(1+B_1+B_1^2-B_2)z^{-2}+(1+B_1+B_2+B_1^3-2B_1B_2+B_3)\\
         & +z^{6}+B_1 z^{4}+(1+B_1+B_1^2-B_2)z^{2}
\end{align*}
and the number of canonical Hodge classes of $H^k(X,\QQ)$
corresponds to the constant term in the coefficient of $z^k$.
\end{example}
\begin{remark}
The computation in example \ref{eg:char} is reversible.
Explicitly, the number of copies of the highest-weight
representation $V_X(\lambda)$ appearing in $\VV_{n,k}$ corresponds
to the coefficient of $\Char(V_X(\lambda))$
\end{remark}
\begin{proposition}\label{prop:charGX}
The character of the $G_X$ representation is 
\begin{equation}\label{eq:pg(z,t)}
p_g(z,t):=\sum_{n=0}^{\infty}p(z)_n
t^n=\prod_{m=1}^{\infty}\frac{1}{\det(I_{24}-gt^m)}\\
\end{equation}
where $g\in T_X$ the maximal torus of $G_X$,
$I_N$ is a $N\times N$ identity matrix, and $\det(I_{24}-g
t^m)=(1-z^{-2}t^m)\cdot(1-z^2t^m)\det(I_{22}-g|_{T_S}t^m)$.
\end{proposition}
\begin{proof}
Recall that $\VV=H^*(S,\QQ)[2]=H^0(S,\QQ)+H^2(S,\QQ)+H^4(S,\QQ)$ is
bigraded, where $H^0(S,\QQ)$ and $H^4(S,\QQ)$ are of weight $-2$ and
$+2$ respectively, and $H^2(S,\QQ)$ is of weight 0. For every symmetric power
$\Sym^k(\VV)$, it is sufficient to show the formula holds when $g$ is
diagonal. Since $G_S$ acts trivially on $H^0$ and $H^4$, let
$u_{-2}\in H^0(S,\QQ)$ be the eigenvector with eigenvalue 1 and weight
$-2$, and $u_2\in H^4(S,\QQ)$ be eigenvector with eigenvalue $1$ and
weight $2$.
 $H^2(S,\QQ)$ corresponds to the standard $G_S$ representation $V_S(1,0,\dots,0)$,
  and let $v_i~(i = 1,\dots,22)$ be
 its eigenvectors; when $i$ is even, $v_i$ has eigenvalue $t_{\frac{i}{2}}$,
  and when $i$ is odd, $v_i$ has eigenvalue
 $t_{\lfloor\frac{i}{2}\rfloor}^{-1}$. \\
Molien's Formula in \cite{St2010}  indicates that for a
representation $W$ of a group $G$, and given a linear operator $g\in
G$, its action on the symmetric algebra $\Sym^\bullet(W)$ has the graded
character
$$\sum_{i=0}^{\infty}\Char(\Sym^i(g))t^i=\frac{1}{\det(I-tg)}~.$$
The symmetric algebra on $\VV$, denoted by
 $\Sym^\bullet(\VV)$, has the form
 $$\Sym^\bullet(u_2)\otimes \Sym^\bullet(v_1)\otimes\Sym^\bullet(v_2)\otimes\cdots\otimes \Sym^\bullet(v_{22})\otimes\Sym^\bullet(u_{-2})~.$$
 Since $v_i\in \VV$ is of weight 0,
 the bigraded character of $\Sym^\bullet (v_i)$ is
  $$\sum_{k=0}^{\infty}(\mu_iz^0)^k\cdot t^k=\frac{1}{1-\mu_i t}~,$$
   where $\mu_i$ is the eigenvalue
 of $v_i$. Since $u_{-2}$ (resp. $u_2$) has weight $-2$ (resp. weight $+2$) in $\VV$, we have
\begin{align*}
 \Char(\Sym^\bullet(u_{-2}))(t) &=\sum_{k=0}^{\infty}(1\cdot z^{-2})^k t^k=\frac{1}{1-z^{-2}t} ~\textrm{, and }\\
 \Char(\Sym^\bullet(u_{2}))(t) &=\sum_{k=0}^{\infty}(1\cdot z^{2})^k t^k=\frac{1}{1-z^{2}t}~.
\end{align*}
  Thus, we obtain
  \begin{align}\label{eq:1sym}
  &\Char(\Sym^\bullet(\VV))(t)
            = \frac{1}{(1-z^2t)(1-z^{-2}t)\det(I_{22}-g|_{T_S}t)}~.
  \end{align}
By Theorem \ref{thm:lehnsoger}, the graded character $p(z)_n$ of
$H^*(X,\QQ)$ is given by
\begin{align*}
p(z)_n t^n= & \Char\left(\sum_{\|\alpha\|=n}\bigotimes_i\Sym^{\alpha_i}\VV  \right)t^n \\
       = & \sum_{\|\alpha\|=n}\prod_{i=1}^{n}\left(\Char\left(\Sym^{\alpha_i}\VV \right) ~ t^{i\alpha_i}\right)~.
\end{align*}

Note $\Char\left(\Sym^{\alpha_i}\VV~\right)t^{i\alpha_i}$ is the
$\alpha_i$-th term in $\Char(\Sym^\bullet(\VV))(t^i)$.
%
Then for each
$\alpha=(1^{\alpha_1},2^{\alpha_2},\dots,n^{\alpha_n})$ where
$\parallel\alpha\parallel = 1\alpha_1+\cdots + n\alpha_n = n$,
$$\prod_{i=1}^{n} \left(\Char\left(\Sym^{\alpha_i}\VV~\right)t^{i\alpha_i}
\right)$$ corresponds to the $t^n$-th term in
$$\prod_{j=1}^{n} \left(\Char(\Sym^\bullet(\VV))(t^j) \right)~.$$
Together with equation \ref{eq:1sym}, one can obtain
%
$$\sum_{n=0}^{\infty} p(z)_nt^n=\prod_{m=1}^{\infty}\frac{1}{(1-z^2t^m)(1-z^{-2}t^m)\det(I_{22}-g|_{T_S}t^m)}~.$$

\end{proof}
\begin{remark}
Given a smooth projective complex surface $S'$,  let
$p(S'^{[n]},z)$ be the Poincar{\'e} polynomial
$\sum_{i=0}^{4n}\beta_i(S'^{[n]})z^i$ of $S'^{[n]}$. 
G\"{o}ttsche \cite{got90} shows that $\sum_{n=0}p(S'^{[n]},z)q^n $ has the expression
    \begin{align}\label{eq:gosttche}
         \prod_{m=1}^{\infty}
            \frac{(1+z^{2m-1}q^m)^{b_1(S')}(1+z^{2m+1}q^m)^{b_{1}(S')}}
            {(1-z^{2m-2}q^m)^{b_0(S')}(1-z^{2m}q^m)^{b_{2}(S')}(1-z^{2m+2}q^m)^{b_4(S')}}~.
    \end{align}
For the case when $S'$ is a $K3$ surface, $b_0(S') = b_4(S') = 1$, $b_1(S')=b_3(S')=0$, and 
$b_2(S')=22$. Letting $z^2q=t$, Equation \ref{eq:gosttche} becomes
    \begin{align*}
         \prod_{m=1}^{\infty}
            \frac{1}
            {(1-z^{-2}t^m)(1-t^m)^{22}(1-z^{2}t^m)}~.
    \end{align*}
\noindent This is the same as taking the character of the  identity in Proposition \ref{prop:charGX}.
\end{remark}
%


\end{subsection}


%
%
%
%

\section{Generating Series for the Character of the Middle Cohomology}\label{sec:5}
\noindent \textbf{ Proof of Theorem \ref{thm:middim}:} For $X$ is of  $K3^{[n]}$-type, we consider the middle cohomology $H^{2n}(X,\QQ)$ of $X$. 
The character of $H^{2n}(X,\QQ)$ is of weight $0$ in Equation
\ref{eq:pg(z,t)}  in Proposition \ref{prop:charGX}. 
The character formula  can be written as
\begin{equation}\label{eq:wt}
\left(\prod_{m=1}\frac{(1-q^m)^2}{(1-z^{-2}q^m)(1-z^2q^m)}\right) \left( \prod_{k=1}(1-q^k)^2\cdot\det(I-g|_{T_S}q^k)\right)^{-1}~.
\end{equation}
\noindent Lemma 1 in \cite{AG84} indicates
\begin{equation}\label{eq:hecke}
\prod_{m=1}\frac{(1-q^m)^2}{(1-z^{-2}q^{m-1})(1-z^2q^{m})} = 
      \sum_{N,r =-\infty\atop~r\geq|N|}^{\infty}(-1)^{r+N}z^{2N} q^{\frac{r^2-N^2+r+N}{2}}~.
\end{equation}
\noindent Multiplying  both sides of Equation \ref{eq:hecke} by $(1-z^{-2})$, we obtain
\begin{align*}
    & \prod_{m=1}\frac{(1-q^m)^2}{(1-z^{-2}q^m)(1-z^2q^m)}  \\
=   & \sum_{N,r =-\infty\atop~r\geq|N|}^{\infty}(-1)^{r+N} q^{\frac{r^2-N^2+r+N}{2}} \left(z^{2N}-z^{2(N-1)}\right)\\
=   & \sum_{N=-\infty}^{\infty}\left( 
        \sum_{l=0 \atop r=|N|+l}^{\infty}(-1)^{r+N} q^{\frac{r^2-N^2+r+N}{2}}-\sum_{l=0 \atop r=|N+1|+l}^{\infty} (-1)^{r+N+1}q^{\frac{r^2-(N+1)^2+r+N+1}{2}}\right)
        z^{2N}~.\\
\end{align*}
\noindent Fix $N\geq 0$, the coefficient of $z^{2N}$ is
\begin{equation}\label{eq:akak}
    \sum_{l=1}^{\infty}(-1)^{l+1} q^{\frac{l(l-1+2N)}{2}}-\sum_{l=1}^{\infty} (-1)^{l+1}q^{\frac{l\left(l-1+2(N+1)\right)}{2}}~.
\end{equation}
Let $a_k(q)=\sum_{l=1} (-1)^{l+1} ~ q^{\frac{l(l-1+2k)}{2}}$ for $k\geq 0$,  the coefficient of 
$z^{2k}$  can be written as $a_{k}(q)-a_{k+1}(q)$. In particular,
the coefficient of $z^0$ is $$a_0(q)-a_1(q)=1+\sum_{l=1}2(-1)^l
q^{\frac{l(l+1)}{2}}~.$$  Thus the coefficient of $z^0$ in Equation \ref{eq:wt} is 
$$\frac{1+\sum_{l=1}2(-1)^lq^{\frac{l(l+1)}{2}}}{\det(I-gq^k)}$$
 where
$\det(I-gq^k)=(1-q^k)^2\cdot\det(I-g|_{T_S}q^k)~.$

\begin{corollary} Let $\beta_{i}(S^{[n]})$ denote the $i$-th Betti number of $S^{[n]}$. We have 
$$ \sum_{n=0}^{\infty} \beta_{2n+2k}(S^{[n]}) ~q^n
    = \frac{q}{\Delta(q)}(a_{k}(q)-a_{k+1}(q)),~k\geq 0
$$
where $\Delta(q)=q\prod_m(1-q^m)^{24}$ is a cusp form of weight 12
for $SL_2(\mathbb{Z})$, and $a_k(q)=\sum_{l=1} (-1)^{l+1} ~q^{\frac{l(l-1+2k)}{2}}$.
\end{corollary}
\begin{proof}
 The corollary follows from the proof of Theorem \ref{thm:middim} (see Equation \ref{eq:akak}) by taking the trivial representation.
\end{proof}

\begin{remark} G\"{o}ttsche \cite{got90} shows that the generating series
for the Euler numbers of $S^{[n]}$ is
$$\sum_{n=0}^{\infty}e(S^{[n]})q^n=\frac{q}{\Delta(q)}~.$$
According to \cite{Zagier} and remark 3.7 of Appendix in
\cite{Hirzebruch} we have
\[\Delta(q) ={4096\epsilon(\delta^2-\epsilon)^2}\]
where $\epsilon$ and $\delta$ are modular forms for $\Gamma_0(2)$ of
weights $4$ and $2$ with the following forms:
\[\epsilon = \sum_{n=1}^{\infty} \left( \sum_{d|n, ~\frac{n}{d}\textrm{ odd}} d^3 \right)q^n\]
\[\delta = -\frac{1}{8}-3\sum_{n=1}^{\infty}\left(\sum_{d|n, ~d\textrm{ odd}} d \right)q^n\]\\
It is interesting that the modular form $\Delta$ appears many
times in computations related to the cohomology rings of
$S^{[n]}$.
\end{remark}
%
%
%

%
%
%
%

\appendix
\section{Table of $G_X$ Representations}\label{app:t1}
\noindent The following computations carried  out using MAGMA \cite{Magma}. Table \ref{table:GXAlL} records $H^*(X,\QQ)$ as a decomposition of $G_X$ representations,
where $X$ is of $K3^{[n]}$-type. $\lambda$ in each row denotes the highest weight of the $G_X$ representation,  
and $\VV_{k,n}$ in each column denotes $H^k(X,\QQ)$.
 Each integer datum indicates the number of copies of the highest-weight representation $V_X(\lambda)$ in  $H^k(X,\QQ)$.
\begin{landscape}
\begin{table}[h]
  \resizebox{1\textwidth}{!}{\begin{minipage}{\textwidth}
\begin{tabular}{| l | l | l | l | l | l | l | l | l | l | l | l | l | l | l | l | l}
    \hline
    \tiny$\lambda$    & \tiny $\VV_{8,7}$ & \tiny $\VV_{10,7}$  & \tiny $\VV_{12,7}$  & \tiny $\VV_{14,7}$  & \tiny $\VV_{8,8}$  & \tiny $\VV_{10,8}$   & \tiny $\VV_{12,8}$  & \tiny $\VV_{14,8}$  & \tiny $\VV_{16,8}$  & \tiny $\VV_{10,9}$  & \tiny $\VV_{12,9}$  & \tiny $\VV_{14,9}$  & \tiny $\VV_{16,9}$  & \tiny $\VV_{18,9}$  & \tiny\tiny ...\\ \hline 
    \tiny$(0,0,0,..)$ & \tiny 5     & \tiny 5       & \tiny 10      & \tiny 7       & \tiny 6          & \tiny 6        & \tiny 13      & \tiny 12      & \tiny 18      & \tiny 6       & \tiny 15      & \tiny 15      & \tiny 25      & \tiny 21      & \tiny\tiny ...\\ \hline 
    \tiny$(1,0,0,..)$ & \tiny 5     & \tiny 9       & \tiny 11      & \tiny 14      & \tiny 5          & \tiny 10       & \tiny 14      & \tiny 21      & \tiny 21      & \tiny 11      & \tiny 16      & \tiny 27      & \tiny 33      & \tiny 39      & \tiny\tiny ...  \\ \hline 
    \tiny$(2,0,0,..)$ & \tiny 4     & \tiny 5       & \tiny 10      & \tiny 9       & \tiny 4          & \tiny 6        & \tiny 13      & \tiny 15      & \tiny 21      & \tiny 6       & \tiny 14      & \tiny 19      & \tiny 31      & \tiny 30      & \tiny\tiny ...  \\ \hline 
    \tiny$(1,1,0,..)$ & \tiny 1     & \tiny 4       & \tiny 4       & \tiny 7       & \tiny 1          & \tiny 4        & \tiny 5       & \tiny 10      & \tiny 9       & \tiny 4       & \tiny 6       & \tiny 13      & \tiny 15      & \tiny 21      & \tiny\tiny ... \\ \hline 
    \tiny$(3,0,0,..)$ & \tiny 1     & \tiny 4       & \tiny 5       & \tiny 7       & \tiny 1          & \tiny 4        & \tiny 6       & \tiny 11      & \tiny 11      & \tiny 4       & \tiny 7       & \tiny 14      & \tiny 18      & \tiny 24      & \tiny\tiny ... \\ \hline 
    \tiny$(2,1,0,..)$ & \tiny 1     & \tiny 2       & \tiny 4       & \tiny 5       & \tiny 1          & \tiny 2        & \tiny 5       & \tiny 8       & \tiny 10      & \tiny 2       & \tiny 5       & \tiny 10      & \tiny 16      & \tiny 18      & \tiny\tiny ...  \\ \hline 
    \tiny$(1,1,1,..)$ & \tiny       & \tiny         & \tiny 1       & \tiny         & \tiny            & \tiny          & \tiny 1       & \tiny 1       & \tiny 2       & \tiny         & \tiny 1       & \tiny 1       & \tiny 3       & \tiny 3       & \tiny\tiny ... \\ \hline 
    \tiny$(4,0,0,..)$ & \tiny 1     & \tiny 1       & \tiny 3       & \tiny 3       & \tiny 1          & \tiny 1        & \tiny 4       & \tiny 5       & \tiny 8       & \tiny 1       & \tiny 4       & \tiny 6       & \tiny 12      & \tiny 11      & \tiny\tiny ... \\ \hline 
    \tiny$(3,1,0,..)$ & \tiny       & \tiny 1       & \tiny 2       & \tiny 3       & \tiny            & \tiny 1        & \tiny 2       & \tiny 5       & \tiny 5       & \tiny 1       & \tiny 2       & \tiny 6       & \tiny 9       & \tiny 13      & \tiny\tiny ... \\ \hline 
    \tiny$(2,2,0,..)$ & \tiny       & \tiny         & \tiny 1       & \tiny         & \tiny            & \tiny          & \tiny 1       & \tiny 1       & \tiny 3       & \tiny         & \tiny 1       & \tiny 1       & \tiny 4       & \tiny 3       & \tiny\tiny ... \\ \hline 
    \tiny$(2,1,1,..)$ & \tiny       & \tiny         & \tiny         & \tiny 1       & \tiny            & \tiny          & \tiny         & \tiny 1       & \tiny 1       & \tiny         & \tiny         & \tiny 1       & \tiny 2       & \tiny 3       & \tiny\tiny ... \\ \hline 
    \tiny$(5,0,0,..)$ & \tiny       & \tiny 1       & \tiny 1       & \tiny 2       & \tiny            & \tiny 1        & \tiny 1       & \tiny 3       & \tiny 3       & \tiny 1       & \tiny 1       & \tiny 4       & \tiny 5       & \tiny 8       & \tiny\tiny ... \\ \hline 
    \tiny$(4,1,0,..)$ & \tiny       & \tiny         & \tiny 1       & \tiny 1       & \tiny            & \tiny          & \tiny 1       & \tiny 2       & \tiny 3       & \tiny         & \tiny 1       & \tiny 2       & \tiny 5       & \tiny 6       & \tiny\tiny ... \\ \hline 
    \tiny$(3,2,0,..)$ & \tiny       & \tiny         & \tiny         & \tiny 1       & \tiny            & \tiny          & \tiny         & \tiny 1       & \tiny 1       & \tiny         & \tiny         & \tiny 1       & \tiny 2       & \tiny 3       & \tiny\tiny ... \\ \hline 
    \tiny$(3,1,1,..)$ & \tiny       & \tiny         & \tiny         & \tiny         & \tiny            & \tiny          & \tiny         & \tiny         & \tiny 1       & \tiny         & \tiny         & \tiny         & \tiny 1       & \tiny 1       & \tiny\tiny ... \\ \hline 
    \tiny$(2,2,1,..)$ & \tiny       & \tiny         & \tiny         & \tiny         & \tiny            & \tiny          & \tiny         & \tiny         & \tiny         & \tiny         & \tiny         & \tiny         & \tiny         & \tiny 1       & \tiny\tiny ... \\ \hline 
    \tiny$(6,0,0,..)$ & \tiny       & \tiny         & \tiny 1       & \tiny         & \tiny            & \tiny          & \tiny 1       & \tiny 1       & \tiny 2       & \tiny         & \tiny 1       & \tiny 1       & \tiny 3       & \tiny 3       & \tiny\tiny ... \\ \hline 
    \tiny$(5,1,0,..)$ & \tiny       & \tiny         & \tiny         & \tiny 1       & \tiny            & \tiny          & \tiny         & \tiny 1       & \tiny 1       & \tiny         & \tiny         & \tiny 1       & \tiny 2       & \tiny 3       & \tiny\tiny ... \\ \hline 
    \tiny$(4,2,0,..)$ & \tiny       & \tiny         & \tiny         & \tiny         & \tiny            & \tiny          & \tiny         & \tiny         & \tiny 1       & \tiny         & \tiny         & \tiny         & \tiny 1       & \tiny 1       & \tiny\tiny ... \\ \hline 
    \tiny$(4,1,1,..)$ & \tiny       & \tiny         & \tiny         & \tiny         & \tiny            & \tiny          & \tiny         & \tiny         & \tiny         & \tiny         & \tiny         & \tiny         & \tiny         & \tiny 1       & \tiny\tiny ... \\ \hline 
    \tiny$(3,3,0,..)$ & \tiny       & \tiny         & \tiny         & \tiny         & \tiny            & \tiny          & \tiny         & \tiny         & \tiny         & \tiny         & \tiny         & \tiny         & \tiny         & \tiny 1       & \tiny\tiny ... \\ \hline 
    \tiny$(7,0,0,..)$ & \tiny       & \tiny         & \tiny         & \tiny 1       & \tiny            & \tiny          & \tiny         & \tiny 1       & \tiny         & \tiny         & \tiny         & \tiny 1       & \tiny 1       & \tiny 2       & \tiny\tiny ... \\ \hline 
    \tiny$(6,1,0,..)$ & \tiny       & \tiny         & \tiny         & \tiny         & \tiny            & \tiny          & \tiny         & \tiny         & \tiny 1       & \tiny         & \tiny         & \tiny         & \tiny 1       & \tiny 1       & \tiny\tiny ... \\ \hline 
    \tiny$(5,2,0,..)$ & \tiny       & \tiny         & \tiny         & \tiny         & \tiny            & \tiny      & \tiny         & \tiny         & \tiny         & \tiny         & \tiny         & \tiny         & \tiny         & \tiny 1       & \tiny\tiny ... \\ \hline 
    \tiny$(8,0,0,..)$ & \tiny       & \tiny         & \tiny         & \tiny         & \tiny            & \tiny      & \tiny         & \tiny         & \tiny 1       & \tiny         & \tiny         & \tiny         & \tiny 1       & \tiny         & \tiny\tiny ... \\ \hline 
    \tiny$(7,1,0,..)$ & \tiny       & \tiny         & \tiny         & \tiny         & \tiny            & \tiny          & \tiny         & \tiny         & \tiny         & \tiny         & \tiny         & \tiny         & \tiny         & \tiny 1       & \tiny\tiny ... \\ \hline 
    \tiny$(9,0,0,..)$ & \tiny       & \tiny         & \tiny         & \tiny         & \tiny            & \tiny          & \tiny         & \tiny         & \tiny         & \tiny         & \tiny         & \tiny         & \tiny         & \tiny 1       & \tiny\tiny ... \\ \\  
\end{tabular}
\caption{{$G_X$ Representations}} 
\label{table:GXAlL}
\end{minipage} }
\end{table}
\end{landscape}
\begin{remark} \cite{NAK99} and \cite{QIN03} discuss the stable cohomology of $S^{[n]}$. 
  Here we only listed data for unstable parts. 
  In general, the ring structure of $H^*(S^{[n]},\QQ)$ is still unknown using irreducible representations. Section 5 in \cite{ME02} 
  gives partial results on relations.
\end{remark}
%
%
%
%
\section{Number of Canonical Hodge Classes}\label{app:t2}
\noindent The following computations were done using MAGMA \cite{Magma}. 
In Table \ref{table:trivial}, each row is indexed by $n$ -- corresponding to $X$ which is of $K3^{[n]}$-type;
each column is indexed by $k$ -- corresponding to the $k$-th cohomology group $H^{k}(X,\QQ)$. 
Since all the odd cohomologies vanish, we only listed the even values of $k$. 
Each integer datum in Table \ref{table:trivial} refers to the number of canonical Hodge classes in  $H^{k}(X,\QQ)$, i.e. the number of copies
trivial $G_X$ representation. \\
\begin{tikzpicture}
 \matrix (mymatrix) [matrix of nodes,
            nodes in empty cells,
text height=1.5ex,
text width=3ex
            ]
{\hline
\begin{scope} \tikz\node[overlay] at (1ex,-0.4ex){\footnotesize n};
\tikz\node[overlay] at (2ex,0.5ex){\footnotesize k}; \end{scope} 
    & 2 & 4 & 6 & 8 & 10 & 12 & 14 & 16 & 18 & 20 & 22 & 24 & 26 & 28 & 30 & 32 &  $\cdots$     \\ \hline
2   & 0 & 1 & \\
3   & 0 & 1 & 1\\
4   & 0 & 2 & 1 & 3\\
5   & 0 & 2 & 1 & 4 & 2  \\
6   & 0 & 2 & 2 & 5 & 4  &  7 \\
7   & 0 & 2 & 2 & 5 & 5  & 10 & 7\\
8   & 0 & 2 & 2 & 6 & 6  & 13 & 12 & 18 \\
9   & 0 & 2 & 2 & 6 & 6  & 15 & 15 & 25 & 21 \\
10  & 0 & 2 & 2 & 6 & 7  & 16 & 18 & 33 & 33 & 43 \\
11  & 0 & 2 & 2 & 6 & 7  & 16 & 20 & 37 & 42 & 61 & 56 \\
12  & 0 & 2 & 2 & 6 & 7  & 17 & 21 & 41 & 51 & 79 & 84 & 104\\
13  & 0 & 2 & 2 & 6 & 7  & 17 & 21 & 43 & 55 & 91 &108 & 146& 138\\
14  & 0 & 2 & 2 & 6 & 7  & 17 & 22 & 44 & 59 & 101& 129& 188& 205& 238\\
15  & 0 & 2 & 2 & 6 & 7  & 17 & 22 & 44 & 61 & 106& 142& 219& 262& 335& 333\\
16  & 0 & 2 & 2 & 6 & 7  & 17 & 22 & 45 & 62 & 110& 152& 244& 312& 432& 480&538 &\\
17  & $\cdots$\\
};
\draw (mymatrix-1-2.north west) -- (mymatrix-17-2.south west);
\draw (mymatrix-1-3.north west) -- (mymatrix-16-3.south west);
\draw (mymatrix-1-4.north west) -- (mymatrix-16-4.south west);
\draw (mymatrix-1-5.north west) -- (mymatrix-16-5.south west);
\draw (mymatrix-1-6.north west) -- (mymatrix-16-6.south west);
\draw (mymatrix-1-7.north west) -- (mymatrix-16-7.south west);
\draw (mymatrix-1-8.north west) -- (mymatrix-16-8.south west);
\draw (mymatrix-1-9.north west) -- (mymatrix-16-9.south west);
\draw (mymatrix-1-10.north west) -- (mymatrix-16-10.south west);
\draw (mymatrix-1-11.north west) -- (mymatrix-16-11.south west);
\draw (mymatrix-1-12.north west) -- (mymatrix-16-12.south west);
\draw (mymatrix-1-13.north west) -- (mymatrix-16-13.south west);
\draw (mymatrix-1-14.north west) -- (mymatrix-16-14.south west);
\draw (mymatrix-1-15.north west) -- (mymatrix-16-15.south west);
\draw (mymatrix-1-16.north west) -- (mymatrix-16-16.south west);
\draw (mymatrix-1-17.north west) -- (mymatrix-16-17.south west);
\draw (mymatrix-1-18.north west) -- (mymatrix-16-18.south west);
\draw (mymatrix-1-1.north west) -- (mymatrix-1-1.south east);
\end{tikzpicture}

\begin{table}[h]
\caption{Number of Canonical Hodge Classes} 
\label{table:trivial}
\end{table}
\begin{remark} For a fixed $k$ and for $n\geq k$, 
	the number of copies of trivial $G_X$ representations stabilizes. This also can be seen 
	by results in stable cohomology of Hilbert schemes of points on $K3$ surfaces (\cite{NAK99}, \cite{QIN03}).
\end{remark}
\begin{remark}
The Beauville-Bogomolov class $\alpha_X$ and the Chern classes of the tangent bundle are canonical Hodge classes, but
	there are more canonical Hodge classes in addition to these. 
\end{remark}
\begin{example} Consider $X$ of $K3^{[8]}$-type. There are 6 canonical Hodge classes in $H^8(X,\QQ)$ according to the table.  On the other hand, there are only 4 canonical Hodge classes which can be expressed in terms of Chern classes and $\alpha_X$, namely 
	 $$c_4(T_X), ~c_2^2(T_X),~ c_2\alpha_X, ~\alpha_X^2~.$$  Future work will be devoted to finding algebraic expression of all canonical Hodge classes i.e. to express these classes as polynomials in Chern classes of certain coherent sheaves.
\end{example}

\bibliographystyle {plain}
\bibliography{IHPK}

\begin{thebibliography}{10}

\bibitem{AG84}
George~E Andrews.
\newblock Hecke modular forms and the kac-peterson identities.
\newblock {\em Transactions of the American Mathematical Society}, pages
  451--458, 1984.

\bibitem{BAJO11}
B.~Bakker and A.~Jorza.
\newblock Lagrangian hyperplanes in holomorphic symplectic varieties.
\newblock {\em arXiv preprint arXiv:1111.0047}, 2011.

\bibitem{BA13}
Benjamin Bakker.
\newblock A classification of lagrangian planes in holomorphic symplectic
  varieties.
\newblock {\em arXiv preprint arXiv:1310.6341}, 2013.

\bibitem{Be83}
Arnaud Beauville.
\newblock Vari\'et\'es {K}\"ahleriennes dont la premi\`ere classe de {C}hern
  est nulle.
\newblock {\em J. Differential Geom.}, 18(4):755--782 (1984), 1983.

\bibitem{BO86}
Ciprian Borcea.
\newblock Diffeomorphisms of a {$K3$} surface.
\newblock {\em Math. Ann.}, 275(1):1--4, 1986.

\bibitem{Magma}
Wieb Bosma, John Cannon, and Catherine Playoust.
\newblock The {M}agma algebra system. {I}. {T}he user language.
\newblock {\em J. Symbolic Comput.}, 24(3-4):235--265, 1997.
\newblock Computational algebra and number theory (London, 1993).

\bibitem{DE95}
Eduardo Cattani, Pierre Deligne, and Aroldo Kaplan.
\newblock On the locus of {H}odge classes.
\newblock {\em J. Amer. Math. Soc.}, 8(2):483--506, 1995.

\bibitem{rep}
William Fulton and Joe Harris.
\newblock {\em Representation theory, a first course}, volume 129 of {\em
  Graduate Texts in Mathematics}.
\newblock Springer-Verlag, New York, 1991.

\bibitem{got90}
Lothar G{\"o}ttsche.
\newblock The {B}etti numbers of the {H}ilbert scheme of points on a smooth
  projective surface.
\newblock {\em Math. Ann.}, 286(1-3):193--207, 1990.

\bibitem{GR12}
Mark Green, Phillip Griffiths, and Matt Kerr.
\newblock {\em Mumford-{T}ate groups and domains}, volume 183 of {\em Annals of
  Mathematics Studies}.
\newblock Princeton University Press, Princeton, NJ, 2012.

\bibitem{Br10}
David Harvey, Brendan Hassett, and Yuri Tschinkel.
\newblock Characterizing projective spaces on deformations of {H}ilbert schemes
  of {K}3 surfaces.
\newblock {\em Comm. Pure Appl. Math.}, 65(2):264--286, 2012.

\bibitem{HaTs09}
Brendan Hassett and Yuri Tschinkel.
\newblock Moving and ample cones of holomorphic symplectic fourfolds.
\newblock {\em Geom. Funct. Anal.}, 19(4):1065--1080, 2009.

\bibitem{HaTs10}
Brendan Hassett and Yuri Tschinkel.
\newblock Intersection numbers of extremal rays on holomorphic symplectic
  varieties.
\newblock {\em Asian J. Math.}, 14(3):303--322, 2010.

\bibitem{Hirzebruch}
Friedrich Hirzebruch, Thomas Berger, and Rainer Jung.
\newblock {\em Manifolds and modular forms}.
\newblock Aspects of Mathematics, E20. Friedr. Vieweg \& Sohn, Braunschweig,
  1992.
\newblock With appendices by Nils-Peter Skoruppa and by Paul Baum.

\bibitem{St2010}
Stavros Kousidis.
\newblock A closed character formula for symmetric powers of irreducible
  representations.
\newblock In {\em 22nd {I}nternational {C}onference on {F}ormal {P}ower
  {S}eries and {A}lgebraic {C}ombinatorics ({FPSAC} 2010)}, Discrete Math.
  Theor. Comput. Sci. Proc., AN, pages 833--844. Assoc. Discrete Math. Theor.
  Comput. Sci., Nancy, 2010.

\bibitem{LS03}
Manfred Lehn and Christoph Sorger.
\newblock The cup product of {H}ilbert schemes for {$K3$} surfaces.
\newblock {\em Invent. Math.}, 152(2):305--329, 2003.

\bibitem{QIN03}
Wei-Ping Li, Zhenbo Qin, and Weiqiang Wang.
\newblock Stability of the cohomology rings of {H}ilbert schemes of points on
  surfaces.
\newblock {\em J. Reine Angew. Math.}, 554:217--234, 2003.

\bibitem{ME02}
Eyal Markman.
\newblock Generators of the cohomology ring of moduli spaces of sheaves on
  symplectic surfaces.
\newblock {\em J. Reine Angew. Math.}, 544:61--82, 2002.

\bibitem{Ma08}
Eyal Markman.
\newblock On the monodromy of moduli spaces of sheaves on {$K3$} surfaces.
\newblock {\em J. Algebraic Geom.}, 17(1):29--99, 2008.

\bibitem{MA10}
Eyal Markman.
\newblock Integral constraints on the monodromy group of the hyper{K}\"ahler
  resolution of a symmetric product of a {$K3$} surface.
\newblock {\em Internat. J. Math.}, 21(2):169--223, 2010.

\bibitem{MA12}
Eyal Markman.
\newblock The beauville-bogomolov class as a characteristic class.
\newblock {\em arXiv preprint arXiv:1105.3223}, 2011.

\bibitem{Ma11}
Eyal Markman.
\newblock A survey of {T}orelli and monodromy results for
  holomorphic-symplectic varieties.
\newblock In {\em Complex and differential geometry}, volume~8 of {\em Springer
  Proc. Math.}, pages 257--322. Springer, Heidelberg, 2011.

\bibitem{Heisen}
Hiraku Nakajima.
\newblock Heisenberg algebra and {H}ilbert schemes of points on projective
  surfaces.
\newblock {\em Ann. of Math. (2)}, 145(2):379--388, 1997.

\bibitem{NAK99}
Hiraku Nakajima.
\newblock {\em Lectures on Hilbert schemes of points on surfaces}.
\newblock Number~18. American Mathematical Soc., 1999.

\bibitem{PE03}
C.~A.~M. Peters and J.~H.~M. Steenbrink.
\newblock Monodromy of variations of {H}odge structure.
\newblock {\em Acta Appl. Math.}, 75(1-3):183--194, 2003.
\newblock Monodromy and differential equations (Moscow, 2001).

\bibitem{Ver95}
Mikhail~Sergeevic Verbitsky.
\newblock {\em Cohomology of compact hyperkaehler manifolds}.
\newblock ProQuest LLC, Ann Arbor, MI, 1995.
\newblock Thesis (Ph.D.)--Harvard University.

\bibitem{Zagier}
Don Zagier.
\newblock Note on the {L}andweber-{S}tong elliptic genus.
\newblock In {\em Elliptic curves and modular forms in algebraic topology
  ({P}rinceton, {NJ}, 1986)}, volume 1326 of {\em Lecture Notes in Math.},
  pages 216--224. Springer, Berlin, 1988.

\bibitem{ZY83}
Yu.~G. Zarhin.
\newblock Hodge groups of {$K3$} surfaces.
\newblock {\em J. Reine Angew. Math.}, 341:193--220, 1983.

\end{thebibliography}

\end{document}